\documentclass[a4paper,11pt]{article}
\renewcommand{\o}{\vee}

\newcommand{\y}{\wedge}

\usepackage{amsmath}

\newcommand{\romb}{\diamond}

\newcommand{\0}{\Delta}

\DeclareMathOperator{\CON}{\mathrm{Con}}

\newcommand{\variedad}[1]{\mathcal{#1}}

\newcommand{\V}{\variedad{V}}

\newcommand{\id}{\approx}

\renewcommand{\id}{\approx}
\newcommand{\func}{\rightarrow}

\newcommand{\iso}{\cong}

\usepackage{amsthm}
\usepackage{verbatim}
\usepackage[all]{xy}
\newlength{\ancho}
\textwidth=\ancho
\newcommand{\impl}{\rightarrow}

\renewcommand{\>}{\rangle}
\newcommand{\ent}{\Rightarrow}
\newcommand{\tne}{\Leftarrow}

\renewcommand{\phi}{\varphi}
\newcommand{\phis}{{\varphi^*}}

\renewcommand{\th}{\theta}

\newcommand{\ths}{{\theta^*}}
\newcommand{\Cg}{\mathrm{Cg}}
\newcommand{\proj}{\mathrm{pj}}
\renewcommand{\romb}{\times}
\newcommand{\vx}{x,y,z,w}
\newcommand{\vc}{a,b,c,d}
\newcommand{\termth}{\sigma(\vec{X})}
\newcommand{\termths}{\sigma^*(\vec{X})}
\newcommand{\termphi}{\rho(\vec{X})}
\newcommand{\termphis}{\rho^*(\vec{X})}
\newcommand{\largo}[1]{|#1|}
\newtheorem{theorem}{Theorem}
\newtheorem{lemma}[theorem]{Lemma}
\newtheorem{prop}{Proposition}
\newtheorem{corollary}[theorem]{Corollary}

\theoremstyle{definition}

\CompileMatrices

\begin{document}
\title{Boolean Factor Congruences and Property (*)}
\author{Pedro S\'{a}nchez Terraf\thanks{Supported by CONICET}}
\date{}
\maketitle
\begin{abstract}
A variety $\V$ has \emph{Boolean factor congruences (BFC)}  if
the set of factor congruences of every algebra in $\V$ is a distributive
sublattice of its congruence lattice; this property holds in rings
with unit and in every variety which has a semilattice
operation.  BFC has a prominent
role in the study of uniqueness of direct product representations of algebras,
since it is a strengthening of the \emph{refinement property}.

We provide an explicit Mal'cev condition for BFC. With the aid of
this condition, it is shown that 
BFC is equivalent to a variant of the  definability \emph{property (*)}, an open
problem in R. Willard's work~\cite{7}. 
\end{abstract}
\section{Introduction}

There is an extensive research concerning
uniqueness of direct product representations (the book of McKenzie,
McNulty and Taylor~\cite{4} is an excellent reference in the
subject). We may start mentioning the classical  theorem of Wedderburn 
and R. Remak, afterwards generalized by Krull and Schmidt, about direct
representations of groups. 

It is convenient to adopt the language of universal algebra at this
point. An \emph{algebra} is a nonempty set together with an arbitrary but
fixed collection of finitary operations. A \emph{variety} is an
equationally-definable class of algebras over the same language. One
fruitful approach to the problem of uniqueness is given by several 
notions of refinement.
We say that an algebra $A$ has the \emph{refinement 
  property} if  
 for every two direct product
decompositions $A\iso  \prod_i B_i  \iso\prod_j  C_j$, there exist $D_{ij}$ such that $B_i \iso \prod_j D_{ij}$ and 
$C_j \iso \prod_i D_{ij}$. 
 In Figure~\ref{fig:diagrama} (a) we have pictured this situation in
 the case $I = J = \{1,2\}$, where every arrow
correspond to a canonical projection onto a direct factor.
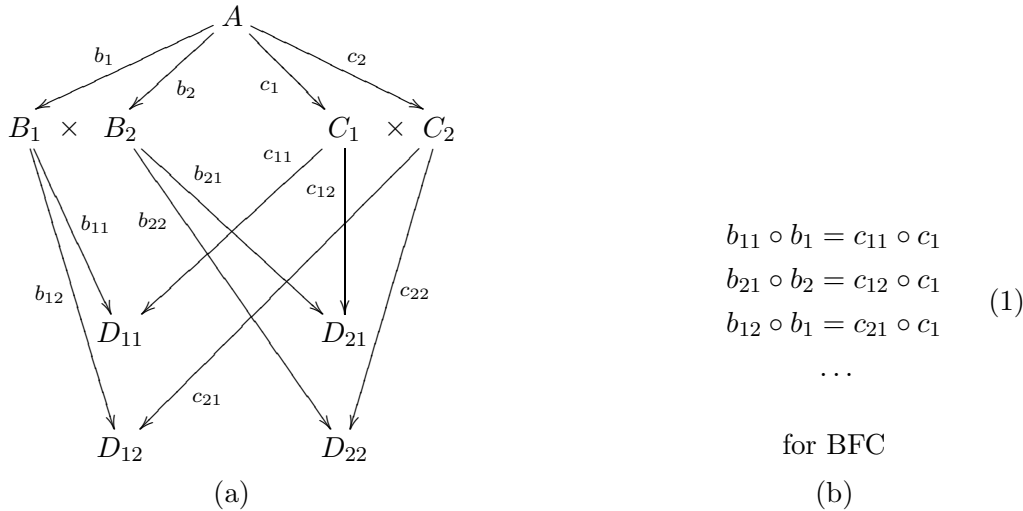
\begin{figure}[h]
\begin{tabular}{ccc}
\raisebox{25ex}[28ex][12ex]{\xymatrix@C=0pt@R=6ex{ & & & A \ar[dlll]!/l 2.5ex/_{b_1} \ar[dl]!/l 1ex/^{b_2}
   \ar[dr]_{c_1} \ar[drrr]!/r 2ex/^{c_2}  \\
  {B_1} \ar[ddrr]^{b_{11}} \ar[dddrr]_{b_{12}} & {\times}  & B_2 \ar[ddrr]^(.3){b_{21}}
  \ar[dddrr]_(.25){b_{22}} &\hspace{5em} & C_1 \ar[ddll]_(.2){c_{11}}
  \ar[dd]_(.3){c_{12}} & \times
  & C_2 \ar[dddllll]^(.8){c_{21}} \ar[dddll]^{c_{22}}  \\
  \\
  & & D_{11} & & D_{21} \\
 & & D_{12} & & D_{22} 
}
}
& \hspace{4em} &
\begin{minipage}{2in}
{\begin{equation}\label{eq:13}
\begin{split}
  b_{11} \circ b_1  &= c_{11} \circ c_1 \\
  b_{21} \circ b_2  &= c_{12} \circ c_1 \\
  b_{12} \circ b_1  &= c_{21} \circ c_1 \\
  &\dots
\end{split}
\end{equation}}

\centerline{for BFC} 
\end{minipage} \\
(a) & & (b)
\end{tabular}
\caption{An instance of refinement and its strict version.\label{fig:diagrama}}
\end{figure}

In \cite{ChaJonTar}, C.~C.~Chang, J\'onsson and Tarski defined Boolean
factor congruences in its full generality and proved it equivalent to a \emph{strict} version of the refinement
property.  A variety $\V$ has \emph{Boolean factor congruences (BFC)}  if
the set of factor congruences of any algebra in $\V$ is a distributive
sublattice of its congruence lattice. Equivalently, if every algebra
in $\V$ satisfy the refinement property with the extra requirement
that the diagram in Figure~\ref{fig:diagrama} is commutative, as
in (b) (see \cite[Theorem 5.6]{ChaJonTar}).

Several years later, D. Bigelow and S. Burris~\cite{1} proved that BFC is a
\emph{Mal'cev property}, and hence  one can assign to every variety
$\V$ with BFC a family
of terms and identities (a \emph{Mal'cev condition}) that ``link''
this property to the syntax of the defining identities of $\V$. 
In our experience,  having an explicit Mal'cev condition may
be very helpful in the search of first-order-logic characterizations
of algebraic concepts. But the result of Bigelow and Burris was based
on Theorem 4.2 of Taylor~\cite{Tay}, which gives a proof using
preservation techniques but does not provide
an explicit Mal'cev condition.

The next step in this direction was taken by Ross
Willard. In his work \cite{7}, he found a very nice definability
property (*) and he proved that it was equivalent to BFC in a broad
class of varieties. A variety $\V$ satisfies \emph{property (*)} if and only if
there exists a factorable\footnote{The definition of \emph{factorable}
  formulas is given in \cite{7}; the main feature of these
  formulas is that they are preserved by taking direct products and factors.} 
first-order  formula  $\pi(x,y,z,w)$ in the  language of $\V$ such that: 
\begin{itemize}
\item $\V\models \pi(x,y,x,y)$
\item $\V\models \pi(x,x,z,w)$
\item $\V\models \pi(x,y,z,z) \impl x=y$
\end{itemize}
That work aimed to obtain a Mal'cev condition for BFC, but only in
2000 Willard found a way to achieve this. He presented his result at
the AMS Spring Southeastern Section Meeting (Columbia, SC). In a
personal communication,
Willard informed D.~Vaggione and the author about this result. He starts
at a property of (not necessarily factor) congruences which is
equivalent to BFC  and then explains a syntactic procedure in order to
produce an explicit Mal'cev condition. However, it appears that a
condition thus generated would be very complicated. 

Here begins the story of this paper. Vaggione and the author were
studying the definability of factor congruences and the
center~\cite{CFC,DFC} and proved that the former implies BFC. 
In the search of an explicit definition, the author pursued the
Mal'cev condition indicated by Willard. From this, a very similar
condition for ``definable factor congruences'' was found. As a
confirmation of our early remark about the role of Mal'cev properties,
we were able to construct a first-order definition $\Phi$ of factor
congruences using  \emph{central elements} (introduced in \cite{va0})
as parameters.

This result was presented in the ``Conference in
Universal Algebra and Lattice Theory'' at Szeged in 2005. During this
conference, Willard asserted that BFC is equivalent property (*) in general,
arguing on the finiteness of the set of terms involved in witnessing
BFC. Soon after that, we realized that a construction
line-by-line analog to that of the formula $\Phi$ 
provides a formula $\pi$ and proves this converse.



In this work we prove:
\begin{theorem}\label{th:principal}
Let $\V$ be a variety. The following are
equivalent:
\begin{enumerate}
\item 
There exists a first-order 
  formula  $\pi(x,y,z,w)$ in the   language of $\V$ which is 
  preserved by direct factors and direct products, and such that: 
  \begin{enumerate}
  \item $\V\models \pi(x,y,x,y)$
  \item $\V\models \pi(x,x,z,w)$
  \item $\V\models \pi(x,y,z,z) \impl x=y$
  \end{enumerate}
\item $\V$ has BFC.
\end{enumerate}
\end{theorem}
Strictly speaking, statement (1) in the theorem  is not property (*)
as stated in~\cite{7}. It remains to be checked if every
sentence having these preservation properties is
factorable. In any case, this definition captures the true
essence of BFC, concerning its relation to preservation by taking
direct factors (see~\cite{CFC,DFC}), and we will keep that name. 

The proof of this theorem will be an application of the results in
\cite{DFC}. In order to do this we will have to restate several
results in that work for the case of BFC. We will do this in
Section~\ref{sec:malcev-condition-bfc}, where the Mal'cev condition
for BFC is obtained. The terms
of this condition are the building blocks for our definition of $\pi$,
carried out in Section~\ref{sec:willards_pi}. Finally, we consider
some (counter)examples in Section~\ref{sec:examples}.

Throughout this paper, the following notation will be used. For $A\in
\mathcal{V}$ and 
$\vec{a},\vec{b}\in A^{n}$,  
$\Cg  
^{A}(\vec{a},\vec{b})$ will denote the congruence generated by the set $%
\{(a_k,b_k):1\leq k\leq n\}$. The symbols
$\nabla$ and $\Delta$ will stand for the universal and trivial
congruence, respectively. We will use  $\th \romb \ths = \0$ in
place of  ``$\th$ and $\ths$ are complementary factor
congruences''. The term algebra (in the language of $\V$) and the
$\V$-free algebra on $X$ will be 
denoted by $T(X)$ and $F(X)$, respectively.
The $i$-th component of an element $a$ in a direct product $\Pi_i A_i$
will be called $a^i = \proj_i(a)$; hence, if $a\in 
A_0\times A_1$,  $a = \<a^0,a^1\>$. If elements $a,b$ of an algebra
$A$ are related by a congruence $\th\in\CON(A)$, we will write
interchangeably $(a,b)\in \th$, $a \,\th\, b$ or $a
\stackrel{\th}{\equiv} b$. This notation  generalizes
to tuples, viz.,  $\vec a
\,\th\, \vec b$ means $(a_i,b_i)\in \th$ for all $i$. 

\section{A Mal'cev Condition for BFC}\label{sec:malcev-condition-bfc}
In this section we will rewrite several combinatorial lemmas
from~\cite{DFC} for the case of BFC. In the first place, we need new
definitions of our former functions $\sigma$, $\sigma^*$, $\rho$ and
$\rho^*$.

Let  $s_i,t_i $ be $(2i+2)$-ary terms (in the language of $\V$) for each
$i=1,\dots,n$ and let $A\in\V$.  Let
$(a,b,c,d,a_1,b_1,\dots,a_n,b_n) \in A^{4+2n}$;
we  define  $\sigma(a,b,c,d,a_1,b_1,\dots,a_n,b_n)$ to be the tuple
$(x,y,z,w,x_1,y_1,\dots,x_n,y_n)$ given  by the following recursion:
\begin{align*}
  x&:= a &  w&:=b \\
  y&:= b   &   x_j&:=s_j(x,y,z,w,x_1,y_1,\dots,x_{j-1},y_{j-1}) \\
  z&:= a & y_j &:= b_j 
\end{align*}
We define $\sigma^*$, $\rho$, $\rho^*$ analogously.
\begin{itemize}
\item  $\sigma^*(a,b,c,d,a_1,b_1,\dots,a_n,b_n) =
  (x,y,z,w,x_1,y_1,\dots,x_n,y_n)$ where:
  \begin{align*}
    x&:= a  & w &:= d \\
    y&:= a  & x_j&:=t_j(x,y,z,w,x_1,y_1,\dots,x_{j-1},y_{j-1}) \\
    z&:= c    & y_j &:= b_j 
  \end{align*}
\item  $\rho(a,b,c,d,a_1,b_1,\dots,a_n,b_n) =
  (x,y,z,w,x_1,y_1,\dots,x_n,y_n)$ where:
\begin{align*}
  x&:= a  & w &:= c \\
  y&:= b   & x_j &:= a_j \\
  z&:= c   & y_j&:=s_j(x,y,z,w,x_1,y_1,\dots,x_{j-1},y_{j-1}) 
\end{align*}
\item  $\rho^*(a,b,c,d,a_1,b_1,\dots,a_n,b_n) =
  (x,y,z,w,x_1,y_1,\dots,x_n,y_n)$ where:
  \begin{align*}
    x&:= a  & w &:= d \\
    y&:= b   & x_j &:= a_j \\
    z&:= c  & y_j&:=t_j(x,y,z,w,x_1,y_1,\dots,x_{j-1},y_{j-1}) 
  \end{align*}
\end{itemize}

In the following we restate the first lemmas in \cite{DFC} for
these new functions:
\begin{lemma}\label{l:igualdad_congr}
For every $(a,b,c,d,a_1,b_1,\dots,a_n,b_n) \in A^{4+2n}$, we have the following
identities:
\begin{multline*}
 \Cg(a,c) \o \Cg(b,d) \o \bigvee_i
\Cg(a_i,s_i(a,b,c,d,a_1,b_1,\dots,a_{i-1},b_{i-1}))= \\
 = \Cg((a,b,c,d,f,a_1,b_1,\dots,a_n,b_n),\sigma(a,b,c,d,f,a_1,b_1,\dots,a_n,b_n))
\end{multline*}
\begin{multline*}
\Cg(a,b) \o \bigvee_i \Cg(a_i,t_i(a,b,c,d, a_1,
b_1,\dots,a_{i-1},b_{i-1})) = \\ =
\Cg((a,b,c,d, a_1, b_1,\dots,a_n,b_n),\sigma^*(a,b,c,d, a_1, b_1,\dots,a_n,b_n))
\end{multline*}
\begin{multline*}
\Cg(c,d) \o \bigvee_i \Cg(b_i,s_i(a,b,c,d, a_1,
b_1,\dots,a_{i-1},b_{i-1})) = \\ = \Cg((a,b,c,d,\dots,a_n,b_n),\rho(a,b,c,d, a_1, b_1,\dots,a_n,b_n))
\end{multline*}
\begin{multline*}
 \bigvee_i  \Cg(b_i,t_i(a,b,c,d, a_1, b_1,\dots,a_{i-1},b_{i-1})) =
 \\=   \Cg((a,b,c,d, a_1, b_1,\dots,a_n,b_n),\rho^*(a,b,c,d, a_1, b_1,\dots,a_n,b_n))
\end{multline*}
\end{lemma}


\begin{corollary}\label{l:recursion_aes}
Given $a,b,c,d\in A$ and $\th,\ths\in\CON(A)$
such that $c \,\th\, a \,\ths\, b \,\th\, d$
and for every $a_i$ and $b_i$ with $i=1,\dots,n$ such that
 \begin{equation}
\begin{split}
s_1(a,b,c,d)\stackrel{\th}{\equiv} &\  a_1 \stackrel{\ths}{\equiv}
t_1(a,b,c,d)\\ 
s_2(a,b,c,d,a_1,b_1)\stackrel{\th}{\equiv} &\  a_2 \stackrel{\ths}{\equiv}
 t_2(a,b,c,d,a_1,b_1)\\ 
& \dots \\
s_{j+1}(a,b,c,d,a_1,b_1,\dots,a_{j},b_{j})\stackrel{\th}{\equiv} & \
a_{j+1}
\stackrel{\;\ths}{\equiv}t_{j+1}(a,b,c,d,a_1,b_1,\dots,a_{j},b_{j}) 
\end{split}
\end{equation}
we have
\begin{equation}\label{eq:4}
t(\sigma(a,b,c,d,a_1,b_1,\dots,a_n,b_n)) 
\stackrel{\th}{\equiv}
t(a,b,c,d,a_1,b_1,\dots,a_n,b_n)  \stackrel{\ths}{\equiv}
t(\sigma^*(a,b,c,d,a_1,b_1,\dots,a_n,b_n)) 
\end{equation}
for every $(2n+4)$-ary term $t$ in the language of $\V$.
\end{corollary}

\begin{corollary}\label{l:recursion_bes}
Suppose $a,b,c,d\in A$, $\phi,\phis\in\CON(A)$
such that $c \,\phi\, d$. If $a_i$ and $b_i$ satisfy
\begin{equation}
\begin{split}
s_1(a,b,c,d)\stackrel{\phi}{\equiv} &\  b_1 \stackrel{\phis}{\equiv}
t_1(a,b,c,d)\\ 
& \dots \\
s_{j+1}(a,b,c,d,a_1,b_1,\dots,a_{j},b_{j})\stackrel{\phi}{\equiv} & \ b_{j+1}
 \stackrel{\;\phis}{\equiv}t_{j+1}(a,b,c,d,a_1,b_1,\dots,a_{j},b_{j}) 
\end{split}
\end{equation}
we obtain
\begin{equation}\label{eq:5}
 t(\rho(a,b,c,d,a_1,b_1,\dots,a_n,b_n)) \stackrel{\phi}{\equiv}
 t(a,b,c,d,a_1,b_1,\dots,a_n,b_n)  \stackrel{\phis}{\equiv}
 t(\rho^*(a,b,c,d,a_1,b_1,\dots,a_n,b_n))
\end{equation}
for every $(2n+4)$-ary term $t$ in the language of $\V$.
\end{corollary}
We will also need the following (Gr\"{a}tzer's)
version of  Mal'cev's key observation on principal congruences.

\begin{lemma}\label{malsev}Let $A$ be any algebra and let $a,b\in A,$
 $\vec{a},\vec{b}\in
A^{n}.$ Then $(a,b)\in \Cg ^{A}(\vec{a},\vec{b})$ if and only if 
there
exist $(n+m)$-ary terms 
$p_{1}(\vec{x},\vec{u}),\dots,p_{k}(\vec{x},\vec{u})$,
with $k$ odd and, $\vec{u}\in A^{m}$ such that: 
\begin{equation*}
\begin{split}
a& =p_{1}(\vec{a},\vec{u}) \\
p_{i}(\vec{b},\vec{u})& =p_{i+1}(\vec{b},\vec{u}),\ i\text{
odd} \\
p_{i}(\vec{a},\vec{u})& =p_{i+1}(\vec{a},\vec{u}),\ i\text{
even} \\
p_{k}(\vec{b},\vec{u})& =b
\end{split}
\end{equation*}
The formula $\xi(x,y,\vec x,\vec y, \vec u)$ given by 
\[x = p_1(\vec x, \vec u) \ \y 
\bigwedge_{i \text{ odd}} p_{i}(\vec{y},\vec{u})
=p_{i+1}(\vec{y},\vec{u}) \ \y \;
\bigwedge_{i \text{ even}} p_{i}(\vec{x},\vec{u})
=p_{i+1}(\vec{x},\vec{u}) \ \y \ 
 p_k(\vec y, \vec u) = y\]
is called a \emph{principal congruence formula\footnotemark}.  
\end{lemma}
\footnotetext{It is customary to call ``principal
  congruence formula'' the existential formula  $\exists\vec u\;
  \xi(x,y,\vec x,\vec y, \vec u)$, but we took this license here for
  technical reasons (see the comments after Corollary~\ref{p:def_infinitaria}).}
\begin{corollary}\label{coromalsev}
For every  homomorphism $F:A\rightarrow B$, if $(a,b)\in
\Cg^{A}(\vec{a},\vec{b})$, then $(F(a),F(b))\in \Cg^{B}(F(\vec{a}),F(\vec{b})).$
\end{corollary}
The following theorem gives a Mal'cev condition for BFC. 
We will use $\largo{\alpha}$ to denote the length of a word $\alpha$
and $\varepsilon$ will denote the empty word.
\begin{theorem}\label{th:malcev_bfc}
A variety  $\V$ has BFC if and only if there exist 
integers 
$N=2k$  and $n$, $(2i+2)$-ary terms $s_i$ and $t_i$ for each
$i=1,\dots,n$, and for every word  $\alpha$ in the alphabet $\{1,\dots,N\}$ of
length no greater than $N$ there are terms $L_\alpha, R_\alpha $ such that 
\medskip

\noindent\fbox{$\largo{\alpha}=N$}
\begin{equation}
\begin{split}\label{eq:dfc_N}
L_\alpha(\termphi) & \id R_\alpha(\termphi) \\
L_\alpha(\termphis) & \id R_\alpha(\termphis) \end{split}
\end{equation}
\fbox{$\largo{\alpha}=0$}
\begin{align}
\begin{split}\label{eq:1}
x  & \id L_\varepsilon(\vec X)\\
y & \id R_\varepsilon(\vec X) 
\end{split} \\
L_\varepsilon(\termphi) & \id L_{1}(\termphi)\label{eq:dfc_0_phi_1}\\
R_{j}(\termphi) & \id L_{j+1}(\termphi) \qquad \text{ if
$1\leq j \leq N-1$} \label{eq:dfc_0_phi}\\
R_{ N}(\termphi) & \id R_{\varepsilon}(\termphi) \label{eq:dfc_0_phi_N}
\end{align}
\fbox{$0<\largo{\alpha}<N$}
\smallskip

 If $\largo{\alpha}$ is even then
\begin{align}
L_\alpha(\termphi) & \id L_{\alpha 1}(\termphi)\label{eq:dfc_par_phi_1} \\
R_{\alpha j}(\termphi) & \id L_{\alpha (j+1)}(\termphi)\label{eq:dfc_par_phi} \qquad
\text{ if $1\leq j \leq k-1$} \\
R_{\alpha k}(\termphi) & \id R_{\alpha}(\termphi)\label{eq:dfc_par_phi_k}\\
\begin{split}\label{eq:dfc_par_phi*}
L_\alpha(\termphis) & \id L_{\alpha (k+1)}(\termphis) \\
R_{\alpha j}(\termphis) & \id L_{\alpha (j+1)}(\termphis) \qquad
\text{ if $k+1\leq j \leq N-1$} \\
R_{\alpha N}(\termphis) & \id R_{\alpha}(\termphis)
\end{split}
\end{align}

 If $\largo{\alpha}$ is odd then
\begin{align}
\begin{split}
L_\alpha(\termth) & \id L_{\alpha1}(\termth)\label{eq:2} \\
R_{\alpha j}(\termth) & \id L_{\alpha (j+1)}(\termth) \qquad
\text{ if $1\leq j \leq k-1$} \\
R_{\alpha k}(\termth) & \id R_{\alpha}(\termth)
\end{split}\\
\begin{split}
L_\alpha(\termths) & \id L_{\alpha (k+1)}(\termths)\label{eq:3} \\
R_{\alpha j}(\termths) & \id L_{\alpha
(j+1)}(\termths)
\qquad \text{ if $k+1\leq j \leq N-1$} \\
R_{\alpha N}(\termths) & \id R_{\alpha}(\termths)
\end{split}
\end{align}
where $\vec X =
(x,y,z,w,x_1,y_1,\dots, x_n,y_n)$ and  $\sigma$, $\sigma^*$,
$\rho$ and  $\rho^*$ are defined relative to  $s_i,$ $t_i$,
on $T_\V(\vec X)$. 
\end{theorem}

\begin{proof}
($\tne$) Assume the existence of the terms, and suppose $\phi \romb 
\phis= \0$, $\th \romb \ths = \0$, and 
$a \, \th \, c \,\phi\, d \,\th\, b \,\ths\, a $. By~\cite[Lemma 0.2]{7}, we will
prove BFC in the moment we  see $a \,
\phi \, b$. There exist unique $a_i,b_i$ satisfying the following
relations: 
\begin{equation}\label{eq:recursion}
\begin{split}
s_1(a,b,c,d)\stackrel{\th}{\equiv} &\  a_1 \stackrel{\ths}{\equiv}
t_1(a,b,c,d)\\ 
s_1(a,b,c,d)\stackrel{\phi}{\equiv} &\  b_1 \stackrel{\phis}{\equiv}
t_1(a,b,c,d)\\ 
& \dots \\
s_{j+1}(a,b,c,d,a_1,b_1,\dots,a_{j},b_{j})\stackrel{\th}{\equiv} & \
a_{j+1}
\stackrel{\;\ths}{\equiv}t_{j+1}(a,b,c,d,a_1,b_1,\dots,a_{j},b_{j}) \\
s_{j+1}(a,b,c,d,a_1,b_1,\dots,a_{j},b_{j})\stackrel{\phi}{\equiv} & \ b_{j+1}
\stackrel{\;\phis}{\equiv}t_{j+1}(a,b,c,d,a_1,b_1,\dots,a_{j},b_{j}) 
\end{split}
\end{equation}
Note that their definition combines schemes in
Corollaries~\ref{l:recursion_aes} and~\ref{l:recursion_bes}. So, by
equations~(\ref{eq:4}) and (\ref{eq:5}) we have, taking $t:=L_\alpha,
R_\alpha$: 

\begin{gather}
\begin{split}\label{eq:6}
L_\alpha(\sigma(a,b,c,d,a_1,b_1,\dots,a_n,b_n)) &\stackrel{\th}{\equiv}
L_\alpha(a,b,c,d,a_1,b_1,\dots)  \stackrel{\ths}{\equiv}
L_\alpha(\sigma^*(a,b,c,d,a_1,b_1,\dots,a_n,b_n)) \\
L_\alpha(\rho(a,b,c,d,a_1,b_1,\dots,a_n,b_n)) &\stackrel{\phi}{\equiv}
L_\alpha(a,b,c,d,a_1,b_1,\dots)  \stackrel{\phis}{\equiv}
L_\alpha(\rho^*(a,b,c,d,a_1,b_1,\dots,a_n,b_n))
\end{split}	\\
\begin{split}\label{eq:12}
R_\alpha(\sigma(a,b,c,d,a_1,b_1,\dots,a_n,b_n)) &\stackrel{\th}{\equiv}
R_\alpha(a,b,c,d,a_1,b_1,\dots)  \stackrel{\ths}{\equiv}
R_\alpha(\sigma^*(a,b,c,d,a_1,b_1,\dots,a_n,b_n)) \\
R_\alpha(\rho(a,b,c,d,a_1,b_1,\dots,a_n,b_n)) &\stackrel{\phi}{\equiv}
R_\alpha(a,b,c,d,a_1,b_1,\dots)  \stackrel{\phis}{\equiv}
R_\alpha(\rho^*(a,b,c,d,a_1,b_1,\dots,a_n,b_n))
\end{split}	
\end{gather}
for every $\alpha$. It can be proved  by an inductive argument that 
$L_\alpha(a,b,c,d,a_1,b_1,\dots,a_n,b_n)=R_\alpha(a,b,c,d,a_1,b_1,\dots,a_n,b_n)$ 
for all $\alpha\neq\varepsilon$, and the proof in \cite{DFC} carries
over \emph{mutatis mutandis}. The reader may find very similar arguments to the those
needed to fulfill this part of the proof in Corollary~\ref{co:Willard_terms}.

\noindent ($\ent$) For each set of variables $Y$, define
\begin{align*}
Y^* &:= Y \cup \{x_{p,q} : p,q \in T(Y)\} \cup \{y_{p,q} : p,q \in
T(Y)\}  \\
Y^{0*} &:= Y \\
Y^{(n+1)*} &:= (Y^{n*})^* \\
Y^\infty &:= \bigcup_{n\geq 1} Y^{n*}
\end{align*}
where $x_{p,q}$ and $y_{p,q}$ are new variables. Take $Z:=\{x,y,z,w\}$
and $F:=F(Z^\infty)$. Define the \emph{index} of $p\in T(Z^\infty)$ as $ind(p) = \min\{j:
p\in T(Z^{j*})\}$; it is evident that if $ind(x_{p,q})\leq
ind(x_{r,s})$, neither $p$ nor $q$ can be terms depending on
$x_{r,s}$. The same holds for $ind(x_{p,q})\leq
ind(y_{r,s})$ and symmetrically, and for  $ind(y_{p,q})\leq
ind(y_{r,s})$. 

Take the following congruences  on $F$:\label{def:delta_infty}
\begin{align*}
\th &:= \Cg(x,z) \o \Cg(y,w) \o \bigvee\{\Cg(p,x_{p,q}): p,q \in F\} & \delta_0 & = \epsilon_0 := \0^F \\
\ths &:= \Cg(x,y) \o \bigvee\{\Cg(x_{p,q},q): p,q \in F\} &
\delta_{n+1} &:= (\th \o \epsilon_n)\cap (\ths \o \epsilon_n)\\ 
\phi &:= \Cg(z,w)  \o \bigvee\{\Cg(p,y_{p,q}): p,q \in
F\}&\epsilon_{n+1} &:= (\phi \o \delta_n)\cap (\phis \o \delta_n)
\\ 
\phis &:=  \bigvee\{\Cg(y_{p,q},q): p,q \in
F\}&\delta_\infty &:= \bigvee_{n\geq 0} \delta_n =  \bigvee_{n\geq 0} \epsilon_n.
\end{align*}
By construction, $\phi \circ \phis = \th \circ \ths = \nabla^F$, $x \, \th
\, z \,\phi\, w \,\th\, y \,\ths\, x $. 
Observe that if $(a,b)\in (\phi\o\delta_\infty) \cap
 (\phis\o\delta_\infty)$ then there exists an $n\geq 0$ such that
$(a,b)\in (\phi\o\delta_n) \cap (\phis\o\delta_n)$. But this
congruence is exactly $\epsilon_{n+1}$, hence $(a,b) \in
\epsilon_{n+1} \subseteq \delta_\infty$. We may conclude $(\phi\o\delta_\infty) \cap
 (\phis\o\delta_\infty) = \delta_\infty $. The same happens with
 $\theta$ and $\ths$,
hence 
\[(\phi\o\delta_\infty)/\delta_\infty \romb 
(\phis\o\delta_\infty)/\delta_\infty = \0 \qquad (\th\o\delta_\infty)/\delta_\infty \romb
(\ths\o\delta_\infty)/\delta_\infty = \0\]
in $F/\delta_\infty$. Then, by BFC we have  $(x/\delta_\infty,y/\delta_\infty) \in
(\phi\o\delta_\infty)/\delta_\infty$ and hence  $(x,y) \in
\phi \o \delta_\infty$. We may find an even integer $N=2k$ such that  $(x,y) \in
\phi \circ^{2N} \delta_N^N$, where $\delta_N^N$ is the result of replacing each
occurrence of ``$\o$'' in the definition of $\delta_N$ by $\circ^N$,
the $n$-fold relational product. Now the terms $L_\alpha$ and $R_\alpha$, for $\alpha$ a word of
length at most $N$ in the alphabet $\{1,\dots,N\}$, can be defined
recursively by using this last congruential equation. Details are
analogous to those in~\cite{DFC}.

\end{proof}
In the next results, we keep the notation of Theorem~\ref{th:malcev_bfc}.
\begin{corollary}\label{co:Willard_terms}
A variety has BFC if and only if
there exist integers $N$ and $n$,  $(2i+2)$-ary terms  $s_i$ and $t_i$
for each
$i=1, \dots, n$ such that for all $A\in\V$ and all  $\theta, \ths, \phi,
\phis \in \CON(A)$ the following holds
\begin{equation}\label{eq:39}
\left.\begin{array}{c}
\Cg(\vec X,  \sigma(\vec X)) \subseteq \th \\
\Cg(\vec X,  \sigma^*(\vec X)) \subseteq \ths \\
\Cg(\vec X, \rho(\vec X))\subseteq \phi \\
\Cg(\vec X,  \rho^*(\vec X))\subseteq \phis
\end{array}\right\}\ent (x,y) \in \phi \o \delta_N.
\end{equation}
for all $x, y, z, w, x_1, y_1, \dots, x_n, y_n$ in $A$.
\end{corollary}
\begin{proof}
($\tne$) Suppose $\theta, \ths, \phi,
\phis \in \CON(A)$ satisfy
\begin{align*}
\th \times \ths &= \Delta &  x \,\th\, z \, \phi\,& w \,\th\, y\\
\phi \times \phis &= \Delta & x \, \ths\, & y
\end{align*}
and $(x,y)\in \th$. As we saw in the first part of the proof of 
Theorem~\ref{th:malcev_bfc}, the congruential equations in the 
antecedent of~(\ref{eq:39})  have (unique) solution for the
indeterminates  $x_1, y_1, \dots,
x_n, y_n$. The construction is given by
equations~(\ref{eq:recursion}), and Lemma~\ref{l:igualdad_congr} says
that these equations are the same as those above.

Since $\th\cap \ths = \phi 
\cap \phis = \Delta$,  we have $\delta_N=\Delta$  
and we conclude $(x,y) \in \phi$. 
Hence we proved that the variety has BFC. 

 ($\ent$) Suppose $\V$ has BFC. The integers $N$ and $n$ 
and the terms are provided by
Theorem~\ref{th:malcev_bfc}. 
Thanks to Corollary~\ref{coromalsev},
it suffices to verify the result in the instance given by
$A = F (\vec X)
= F(x,y,z,w, x_1, y_1, \dots, x_n, y_n)$ and the congruences
\begin{align*}
\theta& = \Cg(\vec X, \sigma(\vec X))  & \phi &= \Cg(\vec X, \rho(\vec X))  \\
\ths& = \Cg(\vec X, \sigma^*(\vec X)) & \phis &= \Cg(\vec X, \rho^*(\vec X)).  
\end{align*}
In this context, we will run an inductive argument to show  
that the terms $L_\alpha, R_\alpha$ witness that $(x,y)\in
\phi\o \delta_N$. (This argument is similar to the ($\tne$)-part of the proof of 
Theorem~\ref{th:malcev_bfc}.)

 Take $\alpha$ such that $\largo{\alpha}=N$,
then
\begin{align*}
L_\alpha(x,y,z,w,x_1,y_1,\dots,x_n,y_n)  & \stackrel{\phi}{\equiv} 
L_\alpha(\rho(x,y,z,w,x_1,y_1,\dots,x_n,y_n))
&& \text{by definition of $\phi$}\\
&=  R_{\alpha}(\rho(x,y,z,w,x_1,y_1,\dots,x_n,y_n)) &&  \text{using identities~(\ref{eq:dfc_N})}\\
&  \stackrel{\phi}{\equiv}  R_{\alpha}(x,y,z,w,x_1,y_1,\dots,x_n,y_n) && \text{by
definition of $\phi$}\\
\intertext{And,}
L_\alpha(x,y,z,w,x_1,y_1,\dots,x_n,y_n)  & \stackrel{\;\phis}{\equiv} 
L_\alpha(\rho^*(x,y,z,w,x_1,y_1,\dots,x_n,y_n))
&& \text{by definition of $\phi$}\\
&=  R_{\alpha}(\rho^*(x,y,z,w,x_1,y_1,\dots,x_n,y_n)) &&  \text{using identities~(\ref{eq:dfc_N})}\\
&  \stackrel{\;\phis}{\equiv}  R_{\alpha}(x,y,z,w,x_1,y_1,\dots,x_n,y_n) && \text{by
definition of $\phi$}
\end{align*} 
Hence $\bigl(L_\alpha(x,y,z,w,x_1,y_1,\dots,x_{n},y_{n}),R_\alpha(x,y,z,w,x_1,y_1,\dots,x_{n},y_{n})\bigr)\in
\phi \cap \phis = \epsilon_{1}$ (recall the definition of
$\epsilon_n$ in page~\pageref{def:delta_infty}).

Suppose $\alpha\neq\varepsilon$ 
has odd length $\largo{\alpha}<N$ and assume 
\[L_{\alpha
j}(x,y,z,w,x_1,y_1,\dots,x_{n},y_{n})  \stackrel{\epsilon_{N-\largo{\alpha}}}{\equiv} R_{\alpha j}(x,y,z,w,x_1,y_1,\dots,x_{n},y_{n})\]
 for every
$j=1,\dots,N$. We check that 
\[L_\alpha(x,y,z,w,x_1,y_1,\dots,x_{n},y_{n})
\stackrel{\th\,\o\,\epsilon_{N-\largo{\alpha}}}{\equiv}
R_\alpha(x,y,z,w,x_1,y_1,\dots,x_{n},y_{n}):\]
\begin{align*}
L_\alpha(x,y,z,w,x_1,y_1,\dots,x_n,y_n)  & \stackrel{\th}{\equiv} 
L_\alpha(\sigma(x,y,z,w,x_1,y_1,\dots,x_n,y_n))
&& \text{by definition of $\th$} \\
&= L_{\alpha 1}(\sigma(x,y,z,w,x_1,y_1,\dots,x_n,y_n))
&& \text{by identities~(\ref{eq:2})} \\
 & \stackrel{\th}{\equiv}  L_{\alpha 1}(x,y,z,w,x_1,y_1,\dots,x_n,y_n) &&
\text{by definition of $\th$} \\
&\!\!\!\!\stackrel{\epsilon_{N-\largo{\alpha}}}{\equiv}  R_{\alpha 1}(x,y,z,w,x_1,y_1,\dots,x_n,y_n)
&& \text{by inductive hypothesis} \\
 & \stackrel{\th}{\equiv}  R_{\alpha 1}(\sigma(x,y,z,w,x_1,y_1,\dots,x_n,y_n))
&& \text{by definition of $\th$}  \\
 & \stackrel{\th}{\equiv}  \ \cdots && \text{using~(\ref{eq:2})}\\
& \!\!\!\!\stackrel{\epsilon_{N-\largo{\alpha}}}{\equiv}  \cdots && \text{and iterating\dots}\\ 
& = R_{\alpha k}(\sigma(x,y,z,w,x_1,y_1,\dots,x_n,y_n)) \\
& = R_{\alpha}(\sigma(x,y,z,w,x_1,y_1,\dots,x_n,y_n)) && \text{using
identities~(\ref{eq:2})}\\
&  \stackrel{\th}{\equiv}  R_{\alpha} (x,y,z,w,x_1,y_1,\dots,x_n,y_n),
\end{align*} 
In the same way we
 show
\[L_\alpha(x,y,z,w,x_1,y_1,\dots,x_{n},y_{n})  \stackrel{\ths\,\o\,
  \epsilon_{N-\largo{\alpha}}}{\equiv}
R_\alpha(x,y,z,w,x_1,y_1,\dots,x_{n},y_{n}):\] 
\begin{align*}
L_\alpha(x,y,z,w,x_1,y_1,\dots,x_n,y_n)  & \stackrel{\ths}{\equiv} 
L_\alpha(\sigma^*(x,y,z,w,x_1,y_1,\dots,x_n,y_n))
&& \text{by definition of $\ths$} \\
&= L_{\alpha (k+1)}(\sigma^*(x,y,z,w,x_1,y_1,\dots,x_n,y_n))
&& \text{by identities~(\ref{eq:3})} \\
 & \stackrel{\ths}{\equiv}  L_{\alpha (k+1)}(x,y,z,w,x_1,y_1,\dots,x_n,y_n) &&
\text{by definition of $\ths$} \\
&\!\!\!\!\stackrel{\epsilon_{N-\largo{\alpha}}}{\equiv}  R_{\alpha (k+1)}(x,y,z,w,x_1,y_1,\dots,x_n,y_n)
&& \text{by ind. hypothesis} \\
 & \stackrel{\ths}{\equiv}  R_{\alpha (k+1)}(\sigma^*(x,y,z,w,x_1,y_1,\dots,x_n,y_n))
&& \text{by definition of $\ths$}  \\
 & \stackrel{\ths}{\equiv}  \ \cdots && \text{using~(\ref{eq:3})}\\
& \!\!\!\!\stackrel{\epsilon_{N-\largo{\alpha}}}{\equiv}  \cdots && \text{and iterating\dots}\\ 
& = R_{\alpha N}(\sigma^*(x,y,z,w,x_1,y_1,\dots,x_n,y_n)) \\
& = R_{\alpha}(\sigma^*(x,y,z,w,x_1,y_1,\dots,x_n,y_n)) && \text{using
identities~(\ref{eq:3})}\\
&  \stackrel{\ths}{\equiv}  R_{\alpha} (x,y,z,w,x_1,y_1,\dots,x_n,y_n),
\end{align*} 
and hence we obtain
\[\bigl(L_\alpha(x,y,z,w,x_1,y_1,\dots),R_\alpha(x,y,z,w,x_1,y_1,\dots\bigr)\in
(\th\o \epsilon_{N-\largo{\alpha}}) \cap (\ths\o \epsilon_{N-\largo{\alpha}}) = \delta_{N-\largo{\alpha}+1}\]

Now suppose $\alpha\neq\varepsilon$ has even length $\largo{\alpha}<N$
and assume
\[L_{\alpha
j}(x,y,z,w,x_1,y_1,\dots,x_{n},y_{n})  \stackrel{\delta_{N-\largo{\alpha}}}{\equiv} R_{\alpha j}(x,y,z,w,x_1,y_1,\dots,x_{n},y_{n})\]
 for every
$j=1,\dots,N$. Then
\begin{align*}
L_\alpha(x,y,z,w,x_1,y_1,\dots,x_n,y_n)  & \stackrel{\phi}{\equiv} 
L_\alpha(\rho(x,y,z,w,x_1,y_1,\dots,x_n,y_n))
&& \text{by definition of $\phi$}\\
&= L_{\alpha 1}(\rho(x,y,z,w,x_1,y_1,\dots,x_n,y_n))
&& \text{by identity~(\ref{eq:dfc_par_phi_1})} \\
 & \stackrel{\phi}{\equiv}  L_{\alpha 1}(x,y,z,w,x_1,y_1,\dots,x_n,y_n) &&
\text{by definition of $\phi$} \\
&\!\!\!\!\stackrel{\delta_{N-\largo{\alpha}}}{\equiv}  R_{\alpha 1}(x,y,z,w,x_1,y_1,\dots,x_n,y_n)
&& \text{by inductive hypothesis} \\
 & \stackrel{\phi}{\equiv}  R_{\alpha 1}(\rho(x,y,z,w,x_1,y_1,\dots,x_n,y_n))
&& \text{by definition of $\phi$}\\
 & \stackrel{\phi}{\equiv}  \ \cdots && \text{using~(\ref{eq:dfc_par_phi})}\\ 
&\!\!\!\!\stackrel{\delta_{N-\largo{\alpha}}}{\equiv}  \cdots && \text{and iterating\dots}\\
& = R_{\alpha k}(\rho(x,y,z,w,x_1,y_1,\dots,x_n,y_n)) \\
& = R_{\alpha}(\rho(x,y,z,w,x_1,y_1,\dots,x_n,y_n)) && \text{using
identity~(\ref{eq:dfc_par_phi_k})}\\
&  \stackrel{\phi}{\equiv}  R_{\alpha} (x,y,z,w,x_1,y_1,\dots,x_n,y_n)&& \text{by definition of $\phi$}
\end{align*} 
proves $\bigl(L_\alpha(x,y,z,w,x_1,y_1,\dots,x_{n},y_{n}),R_\alpha(x,y,z,w,x_1,y_1,\dots,x_{n},y_{n})\bigr)\in
\phi\o \delta_{N-\largo{\alpha}}$. We can see analogously (using
$\rho^*$ and identities~(\ref{eq:dfc_par_phi*})) that  
\[\bigl(L_\alpha(x,y,z,w,x_1,y_1,\dots,x_{n},y_{n}),R_\alpha(x,y,z,w,x_1,y_1,\dots,x_{n},y_{n})\bigr)\in
\phis\o \delta_{N-\largo{\alpha}},\] 
therefore
\[\bigl(L_\alpha(x,y,z,w,x_1,y_1,\dots),R_\alpha(x,y,z,w,x_1,y_1,\dots)\bigr)\in (
\phi\o \delta_{N-\largo{\alpha}}) \cap (\phis\o
\delta_{N-\largo{\alpha}}) = \epsilon_{N-\largo{\alpha}+1}\]
Finally, for $\alpha=\varepsilon$, and noting that
$\delta_{N-\largo{\alpha}} =\delta_{N}$, we have:
\begin{align*}
x &=L_\varepsilon(\rho(x,y,z,w,x_1,y_1,\dots,x_n,y_n)) &&  \text{using identities~(\ref{eq:1})}\\ 
&= L_{ 1}(\rho(x,y,z,w,x_1,y_1,\dots,x_n,y_n))
&& \text{by identity~(\ref{eq:dfc_0_phi_1})} \\
 & \stackrel{\phi}{\equiv}  L_{ 1}(x,y,z,w,x_1,y_1,\dots,x_n,y_n) &&
\text{by definition of $\phi$}\\
& \stackrel{\delta_{N}}{\equiv} R_{ 1}(x,y,z,w,x_1,y_1,\dots,x_n,y_n)
&& \text{by inductive hypothesis} \\
 & \stackrel{\phi}{\equiv}  R_{ 1}(\rho(x,y,z,w,x_1,y_1,\dots,x_n,y_n))
&& \text{by definition of $\phi$}\\
 & \stackrel{\phi}{\equiv} \  \cdots && \text{using
  identities~(\ref{eq:dfc_0_phi})}\\
& \stackrel{\delta_{N}}{\equiv}\ \cdots && \text{and iterating\dots}\\
& = R_{N}(\rho(x,y,z,w,x_1,y_1,\dots,x_n,y_n)) \\
& = R_\varepsilon(\rho(x,y,z,w,x_1,y_1,\dots,x_n,y_n)) && \text{using
identity~(\ref{eq:dfc_0_phi_N})}\\
& = y  &&  \text{using identities~(\ref{eq:1})}
\end{align*} 
This proves $(x,y)\in\phi\o\delta_{N}$.
\end{proof}
This corollary is a variant of Willard's original condition. He states 
that a variety has BFC if there exist $n\geq 0$, and terms
\begin{align*}
s_1(x,y,z,w),&\ t_1(x,y,z,w) \\
s_2(x,y,z,w,u_1),&\ t_2(x,y,z,w,u_1) \\
s_3(x,y,z,w,u_1,u_2),&\ t_2(x,y,z,w,u_1,u_2) \\
& \vdots \\
s_n(x,y,z,w,u_1,\dots,u_{n-1}),&\ t_n(x,y,z,w,u_1,\dots,u_{n-1}) 
\end{align*}
such that $\forall A\in \V, \forall \th, \ths, \phi, \phis \in
\CON(A), \forall a,b,c,d,e_1,\dots,e_n\in A$, if $a
\stackrel{\th}{\equiv} c  \stackrel{\phi}{\equiv} d
\stackrel{\th}{\equiv} b  \stackrel{\ths}{\equiv} a$ and
\begin{align*}
s_i(a,b,c,d,e_1,\dots,e_{i-1}) & \stackrel{\th}{\equiv} e_i
\stackrel{\ths}{\equiv} t_i(a,b,c,d,e_1,\dots,e_{i-1}) \text{ ($1\leq i
\leq n$, $i$ odd)} \\
s_i(a,b,c,d,e_1,\dots,e_{i-1}) & \stackrel{\phi}{\equiv} e_i
\stackrel{\phis}{\equiv} t_i(a,b,c,d,e_1,\dots,e_{i-1}) \text{ ($2\leq i
\leq n$, $i$ even)}
\end{align*}
then $(a,b) \in \phi \o \delta_\infty$. 

The procedure of using  $\delta_\infty$, to force a pair of
congruences in a free algebra freely generated by an 
infinite set to be  factor complementary, already appears as part of
Vaggione's work on Boolean-representable varieties~\cite{va0}. 

In the next corollary, we obtain an infinitary ``formula'' which is
our first approximation to $\pi$.
\begin{corollary}\label{p:def_infinitaria}
Let $A = A_0 \times A_1$  be  an algebra in a variety with BFC, and let
$\Pi(x,y,z,w)$ be the following predicate:
\begin{equation}\label{eq:41}
\exists x_1 \forall y_1 \dots \exists x_n \forall y_n  \ \ \Cg^A(\vec
 X, \sigma(\vec X))  \cap \Cg^A(\vec X, \sigma^*(\vec X))  = \Delta^A  
 \end{equation}
Then, for all $a,b\in A_0$ and $a',b',c'\in A_1$,
$\Pi(\<a,a'\>,\<b,b'\>,\<a,c'\>,\<b,c'\>)$ holds in
$A$ if and only if $a' = b'$.
\end{corollary}
\begin{proof}
We will need to do the following definitions:
\begin{align*}
  x &:=\<a,a'\>  & y &:=\<b,b'\>   \\
  z &:= \<a,c'\> & w &:=\<b,c'\>,
\end{align*}
hence we have $ \Cg(x,z)\o \Cg(y,w) \subseteq \ker\proj_0 $  and
$\Cg(z,w)\subseteq \ker\proj_1$.

($\tne$) Suppose $(x,y) \in \ker \proj_1$. Take $x_1$ such that
 \[s_1(x, y, z, w)\stackrel{\ker\proj_0}{\equiv} \ x_1 \stackrel{\ker\proj_1}{\equiv}
t_1(x, y, z, w)\]
and assuming $x_i$ has already been chosen  and $y_i$ is given, let
\begin{equation*}
s_{i+1}(x, y, z, w, x_1, y_1,\dots, x_i, y_i)\stackrel{\ker\proj_0}{\equiv} \ x_{i+1}
\stackrel{\;\ker\proj_1}{\equiv}t_{i+1}(x, y, z, w, x_1, y_1, \dots, x_i,
y_{i}). 
\end{equation*}
By means of this procedure, and taking into account
Corollary~\ref{l:recursion_aes}, we may conclude that $\Cg^A(\vec
 X, \sigma(\vec X)) \subseteq \ker\proj_0$ and $ \Cg^A(\vec X, \sigma^*(\vec
X)) \subseteq \ker\proj_1$. Since $\ker\proj_0 \cap \ker\proj_1 = \Delta^A$, we have
(\ref{eq:41}).
\smallskip 

($\ent$) Suppose (\ref{eq:41}) holds. Take $y_1$ such that 
\[s_1(x,y,z, w)\stackrel{\ker\proj_1}{\equiv} \ y_1 \stackrel{\ker\proj_0}{\equiv}
t_1(x,y,z, w).\]
(Note: the 
order of congruences is reversed.)
Let $x_1$ given by the outer existential quantifier of 
(\ref{eq:41}). Assuming $y_{i}$  
is already chosen  and $x_i$ is the  corresponding  witness for (\ref{eq:41}), let
\begin{equation*}
s_{i+1}(x,y,z, w,x_1,y_1,\dots,x_{i},y_{i})\stackrel{\ker\proj_1}{\equiv} \ y_{i+1}
\stackrel{\;\ker\proj_0}{\equiv}t_{i+1}(x,y,z,w,x_1,y_1,\dots,x_{i},y_{i}).
\end{equation*}
Corollary~\ref{l:recursion_bes} ensures that  $\Cg^A(\vec
 X, \rho(\vec X)) \subseteq \ker\proj_1$ and $ \Cg^A(\vec X, \rho^*(\vec
X)) \subseteq \ker\proj_0$.

Take in Corollary~\ref{co:Willard_terms}
\begin{align*}
\th& := \Cg(\vec X, \sigma(\vec X))  & \phi &:=  \ker\proj_1 \\
\ths& := \Cg(\vec X, \sigma^*(\vec X)) & \phis &:= \ker\proj_0
\end{align*}
We thus obtain $(x,y)\in \phi\o \delta_N$. Since $ \phi \cap \phis =\ker\proj_1 \cap
 \ker\proj_0 = \Delta^A$ and the same holds for $\th, \ths$, we have
 $\delta_N =\Delta^A$ and hence $(x,y) \in \phi = \ker\proj_1$. This is
 the same to say $a^1 = b^1$.
\end{proof}

Though ``formula''~(\ref{eq:41}) is not in first-order logic, it
corresponds to a formula of the infinitary logic $L_{\kappa^+\omega}$
(here $\kappa$ is the cardinal of the language of $\V$ plus $\omega$), since its ``matrix''  $ \Cg^A(\vec
 X, \sigma(\vec X))  \cap \Cg^A(\vec X, \sigma^*(\vec X))  = \Delta^A$
 can be replaced by an infinite conjunction of quasi-identities. This
 can be seen by considering principal congruence formulas (recall Lemma~\ref{malsev}).
We may write ``$\Cg(\vec a, \vec b) = \Delta$'' in the following fashion:
\[ \bigwedge_{\xi \text{ PCF}} \ \forall x,y \, \forall \vec u_\xi \ \xi(x,y,\vec a, \vec b,
\vec u_\xi) \impl x = y.\]
In the same way,
\[ \bigwedge_{\xi, \; \zeta \text{ PCF}} \ \forall x,y \,\forall \vec u_\xi, \vec v_\zeta\; :\; \xi\bigl(x,y, \vec
X, \sigma(\vec X), \vec u_\xi\bigr)  \y \zeta\bigl(x,y, \vec X,
\sigma^*(\vec X), \vec v_\zeta\bigr)   \impl x=y,\]
is equivalent to  ``$\Cg^A(\vec
 X, \sigma(\vec X))  \cap \Cg^A(\vec X, \sigma^*(\vec X))  =
\Delta^A$''. 

In the next Section  we will see that it is indeed 
possible to find a first-order formula with a similar syntactic structure that
satisfies property (*).

\section{Property (*) and BFC}\label{sec:willards_pi}
Let $\V$ be a variety with BFC. By Theorem~\ref{th:malcev_bfc}, we may
define the following formulas in the language of $\V$:
 \[ \Psi_m := \bigwedge_{\largo{\alpha} = m} 
\Biggl(\Bigl(\bigwedge_{\gamma\neq\varepsilon}
L_{\alpha\gamma}(\vec X) = R_{\alpha\gamma}(\vec X)\Bigr)
\ \impl \ L_\alpha(\vec X) = R_\alpha(\vec X)\Biggr). \]
where every word-subindex moves over  words of length
less than or equal to $N$; so, any expression of the form 
``$\bigwedge_{\gamma\neq\varepsilon} L_{\alpha\gamma} = R_{\alpha\gamma} $'' should be
read as ``$\bigwedge \{ L_{\alpha\gamma} = R_{\alpha\gamma}  : \gamma\neq\varepsilon
\text{ and } \largo{\alpha\gamma}\leq N \}$''. Thus, if
$m> N$, $\Psi_{m} = \mathit{true}$ (empty conjunction) and  $\Psi_N =
\bigl(\bigwedge_{\largo{\beta}=N} L_\beta(\vec X)=R_\beta(\vec X)\bigr)$ (the
antecedent ``vanishes''). 

The formulas $\Psi_m$ will be the building blocks for constructing a
formula $\Phi_2$ that satisfies the elementary requirements of
property (*). But  it is not immediate that $\Phi_2$ will satisfy the
necessary preservation property. Nevertheless, in the context of $\V$
we may prove this. Readily, there is a formula  $ \Phi_1(x,y,z,w)$
valid in $\V$ such that $\Phi_1 \y \Phi_2$ is preserved by direct
products and direct factors.

The following lemma defines $\Phi_1$ and proves its validity over $\V$. 
\begin{lemma}\label{l:deb_q_centro}
Let $\V$ be a variety with BFC. Then 
\begin{equation}\label{eq:deb_q_centro}
\V \models \Phi_1(\vx)\ :=\ 
\exists y_1 \forall x_1  \dots  \exists y_n \forall x_n \bigwedge_{m = 1}^{k}
\Psi_{2m}
\end{equation}
with $n,k$ as in Theorem~\ref{th:malcev_bfc}.  
\end{lemma}
\begin{proof}
Suppose $a,b,c,d\in A\in\V$. Take $b_1 := t_1(a,b,c,d)$. Assuming $b_i$ is
already chosen and $a_{i}$ is given, define 
\[b_{i+1}:= t_{i+1}(\vc,a_1,b_1,\dots,a_{i},b_{i}).\]
The construction of $b_i$'s  ensures 
\begin{equation}\label{eq:9}
(\vc,a_1,b_1,\dots,a_{n},b_{n}) =\rho^*(\vc,a_1,b_1,\dots,a_{n},b_{n}).
\end{equation}
Hence we have
that for each $\beta$ with $\largo{\beta}=N$,
\begin{equation*}
L_\beta(\vc,a_1,b_1,\dots,a_{n},b_{n}) = R_\beta(\vc,a_1,b_1,\dots,a_{n},b_{n}) 
\end{equation*}
by equations~(\ref{eq:dfc_N}), and we conclude $A\models \Psi_N(\vc,a_1,b_1,\dots,a_{n},b_{n})$.  

Take nonempty $\alpha$ with  $0<\largo{\alpha}<N$ even. We will prove
that $\Psi_\alpha$ holds. Suppose 
\[A\models \bigwedge_{\gamma\neq\varepsilon} L_{\alpha\gamma}(\vc,a_1,b_1,\dots,a_{n},b_{n}) =
R_{\alpha\gamma}(\vc,a_1,b_1,\dots,a_{n},b_{n}).\]
or, equivalently,
\begin{equation}\label{eq:10}
A\models \bigwedge_{\gamma\neq\varepsilon} L_{\alpha\gamma}(\rho^*(\vc,a_1,b_1,\dots,a_{n},b_{n})) =
R_{\alpha\gamma}(\rho^*(\vc,a_1,b_1,\dots,a_{n},b_{n})).
\end{equation}
We then have:
\begin{align*}
L_\alpha(\vc,a_1,b_1,\dots) &=
L_\alpha(\rho^*(\vc,a_1,b_1,\dots)) 
&& \text{by equation~(\ref{eq:9})}\\
&= L_{\alpha (k+1)}(\rho^*(\vc,a_1,b_1,\dots)) 
&& \text{by identities~(\ref{eq:dfc_par_phi*})} \\
& = R_{\alpha (k+1)}(\rho^*(\vc,a_1,b_1,\dots)) 
&& \text{by~(\ref{eq:10})} \\
& = \ \cdots && \text{using~(\ref{eq:dfc_par_phi*}),
  (\ref{eq:10})} \\
& = \ \cdots && \text{and iterating\dots}\\
& = R_{\alpha N}(\rho^*(\vc,a_1,b_1,\dots)) \\
& = R_{\alpha}(\rho^*(\vc,a_1,b_1,\dots)) && \text{using
identities~(\ref{eq:dfc_par_phi*})}\\
& = R_{\alpha}(\vc,a_1,b_1,\dots) &&  \text{by equation~(\ref{eq:9}).} 
\end{align*} 
Hence we have
\[A \models   L_\alpha(\vc,a_1,b_1,\dots) =
R_\alpha(\vc,a_1,b_1,\dots),\] 
and we have proved the Lemma.
\end{proof}

\begin{lemma}\label{l:phi_2}
Let $\V$ be a variety with  BFC. Define:
\begin{equation}\label{eq:deb_kernel_dfc}
\Phi_2(x,y,z,w)\ :=\ \exists x_1 \forall y_1 \dots \exists x_n \forall y_n  \bigwedge_{m = 1}^{k}
\Psi_{2m-1}  
\end{equation}
Then $\V\models \Phi_2(x,y,x,y)$ and $\V\models \Phi_2(x,x,z,w)$. 
\end{lemma}
\begin{proof}
We only prove the first one, since the proofs are analogous to that of
the previous lemma.
Suppose $a,b\in A\in\V$. Take $a_1 := s_1(a,b,a,b)$. Assuming $a_i$ is
already chosen and $b_{i}$ is given, define 
\[a_{i+1}:= s_{i+1}(a,b,a,b,a_1,b_1,\dots,a_{i},b_{i}).\]
The construction of $b_i$'s  ensures 
\begin{equation}\label{eq:igual_sigma}
(a,b,a,b,a_1,b_1,\dots,a_{n},b_{n}) =\sigma(a,b,a,b,a_1,b_1,\dots,a_{n},b_{n}).
\end{equation}

Take nonempty $\alpha$ with  $\largo{\alpha}<N$ odd. We will prove
that $\Psi_\alpha$ holds. Suppose 
\[A\models \bigwedge_{\gamma\neq\varepsilon} L_{\alpha\gamma}(a,b,a,b,a_1,b_1,\dots,a_{n},b_{n}) =
R_{\alpha\gamma}(a,b,a,b,a_1,b_1,\dots,a_{n},b_{n}).\]
or, equivalently,
\begin{equation}\label{eq:diez}
A\models \bigwedge_{\gamma\neq\varepsilon} L_{\alpha\gamma}(\sigma(a,b,a,b,a_1,b_1,\dots,a_{n},b_{n})) =
R_{\alpha\gamma}(\sigma(a,b,a,b,a_1,b_1,\dots,a_{n},b_{n})).
\end{equation}
We then have:
\begin{align*}
L_\alpha(a,b,a,b,a_1,b_1,\dots) &=
L_\alpha(\sigma(a,b,a,b,a_1,b_1,\dots)) 
&& \text{by equation~(\ref{eq:igual_sigma})}\\
&= L_{\alpha 1}(\sigma(a,b,a,b,a_1,b_1,\dots)) 
&& \text{by identities~(\ref{eq:2})} \\
& = R_{\alpha 1}(\sigma(a,b,a,b,a_1,b_1,\dots)) 
&& \text{by~(\ref{eq:diez})} \\
& = \ \cdots && \text{using~(\ref{eq:2}),
  (\ref{eq:10})} \\
& = \ \cdots && \text{and iterating\dots}\\
& = R_{\alpha k}(\sigma(a,b,a,b,a_1,b_1,\dots)) \\
& = R_{\alpha}(\sigma(a,b,a,b,a_1,b_1,\dots)) && \text{using
identities~(\ref{eq:2})}\\
& = R_{\alpha}(a,b,a,b,a_1,b_1,\dots) &&  \text{by equation~(\ref{eq:igual_sigma}).} 
\end{align*} 
Hence we have
\[A \models   L_\alpha(a,b,a,b,a_1,b_1,\dots,a_{n},b_{n}) =
R_\alpha(a,b,a,b,a_1,b_1,\dots,a_{n},b_{n}).\] 
The proof that $\V\models \Phi_2(x,x,z,w)$ is similar, but using
$\sigma^*$ and $t_i$'s in place of $\sigma$ and $s_i$'s, respectively.
\end{proof}
\begin{lemma}
Let $a,b,c\in A\in\V$ with BFC. If $A$ satisfies $\Phi_2(a,b,c,c)$, then $a=b$.
\end{lemma}
\begin{proof}
 Assume $A\models \Phi_2(a,b,c,c)$. Take $b_1:=s_1(a,b,c,c)$. Let $a_1$
be given by the outermost existential quantifier of $\Phi_2$.

Assuming $b_{i}$ is
already chosen and $a_i$ is the corresponding witness for $\Phi_2$, let
\begin{equation}\label{eq:22}
b_{i+1}:= s_{i+1}(a,b,c,c,a_1,b_1,\dots,a_{i},b_{i})
\end{equation}
This selection satisfies
\begin{equation}\label{eq:11}
(a,b,c,c,a_1,b_1,\dots,a_{n},b_{n}) =\rho(a,b,c,d,a_1,b_1,\dots,a_{n},b_{n}).
\end{equation}
Using an analogous reasoning to that in the  proof of
Lemma~\ref{l:deb_q_centro} (replacing there $t_i$'s and $\rho^*$ by
$s_i$'s and $\rho$, respectively), the reader may check that this
choice of $a_i$, $b_i$  satisfies the matrix of $\Phi_1(a,b,c,c)$. We
hence obtain
\begin{equation*}
A\models \Bigl(\bigwedge_{m=1}^N \Psi_m\Bigr) (a,b,c,c,a_1,b_1,\dots,a_n,b_n)
\end{equation*}
From an easy inspection of the form of $\Psi_m$, it can be deduced
that
\[A\models \bigwedge_{j=1}^{N} L_j(a,b,c,c,a_1,b_1,\dots,a_{n},b_{n}) =
R_j(a,b,c,c,a_1,b_1,\dots,a_{n},b_{n}),\]
and using~(\ref{eq:11}),
\begin{equation}\label{eq:long_1}
A\models \bigwedge_{j=1}^{N} L_j(\rho(\vc,a_1,b_1,\dots,a_{n},b_{n})) = R_j(\rho(\vc,a_1,b_1,\dots,a_{n},b_{n})).
\end{equation}
Therefore,
\begin{align*}
a & =L_\varepsilon(a,b,c,c,a_1,b_1,\dots,a_{n},b_{n}) && \text{by
  identities~(\ref{eq:1})} \\
&= L_\varepsilon(\rho(\vc,a_1,b_1,\dots,a_{n},b_{n})) && \text{by
  equation~(\ref{eq:11})}\\ 
& = L_1(\rho(\vc,a_1,b_1,\dots,a_{n},b_{n})) &&
\text{by identities~(\ref{eq:dfc_0_phi_1}), with $\alpha=\varepsilon$} \\
&=  R_1(\rho(\vc,a_1,b_1,\dots,a_{n},b_{n})) && \text{by
  equations~(\ref{eq:long_1})}\\  
&= L_2(\rho(\vc,a_1,b_1,\dots,a_{n},b_{n})) && \text{by identities~(\ref{eq:dfc_par_phi})} \\
&= \ \cdots && \text{using
  equations~(\ref{eq:dfc_par_phi}), (\ref{eq:long_1})} \\
&= \ \cdots && \text{and iterating\dots}\\
& = R_N(\rho(\vc,a_1,b_1,\dots,a_{n},b_{n})) && \text{using
equations~(\ref{eq:dfc_par_phi_k}) once more:}\\
& = R_{\varepsilon}(\rho(\vc,a_1,b_1,\dots,a_{n},b_{n})) \\
&= R_\varepsilon(a,b,c,c,a_1,b_1,\dots,a_{n},b_{n}) && \text{by
  equation~(\ref{eq:11})}\\ 
& = b && \text{by identities~(\ref{eq:1})}
\end{align*} 
Hence $a = b$.
\end{proof}
\begin{proof}[Proof of Theorem~\ref{th:principal}]
$(\tne)$ The formula $\pi(x,y,z,w) :=
  \Phi_1(x,y,z,w)\y\Phi_2(x,y,z,w)$ satisfies (a), (b) and (c) in
  Theorem~\ref{th:principal}(1) by the previous lemmas. It is also 
 preserved by taking direct factors and direct products: this is an
 immediate application of \cite[Theorem 22]{DFC}, where we 
take $\vec z = (z,w)$ and $\tau_\alpha(\vec X)$ to be
``$L_\alpha(\vec X)=R_\alpha(\vec X)$''.

$(\ent)$ This is easy to show; for details see \cite[Theorem 1.5]{7}.
\end{proof}

\section{Some (Counter)examples}\label{sec:examples}
One of our main interests was to find an algebraic counterpart of the
formula $\pi$ witnessing  property (*). The first approach is the
characterization in Corollary~\ref{p:def_infinitaria}.  
A second one is given by the following 
 semantic consequence of $\pi$: every time one has $A\models
\pi(a,b,c,d)$, one obtains
\begin{equation}\label{eq:18}
 \text{for every $\th\in FC(A)$, $(c,d) \in \th$ implies $(a,b) \in \th$.}
\end{equation}
where $FC(A)$ is the set of factor congruences of $A$. This can be immediately seen by noting that for all $\th\in FC(A)$ we have  $A/\th\models
\pi(a/\th,b/\th,c/\th,d/\th)$ since $\pi$ is preserved by direct factors, and if $c/\th=d/\th$ we must have $a/\th=b/\th$. 

Now call $\Gamma(a,b,c,d)$ the assertion~(\ref{eq:18}). In spite this
predicate might not be expressible in first-order logic, it can be  proved
that it satisfies all conditions for property (*):

\begin{prop}\label{p:Psi}
For all, $a,b,c,d\in A\in\V$, where $\V$ has BFC, we have:
\begin{enumerate}
\item $\Gamma(a,b,c,d)$ is equivalent to ``$(a,b) \in \bigcap \{\th\in
  FC(A): (c,d) \in \th\}$'', and hence  $\Gamma(a_i,b_i,c,d)$ for all $i=1,\dots,l$ implies
  $\Gamma(F(\vec a),F(\vec b),c,d)$, for every $l$-ary basic operation
  $F$ in the language of $\V$. 
\item $A\models \Gamma(a,a,b,c)$.
\item $A\models \Gamma(a,b,a,b)$.
\item $A\models \Gamma(a,b,c,c) \impl a=b$.
\item If $a',b',c',d'\in B\in\V$, $A\times B\models
  \Gamma(\<a,a'\>,\<b,b'\>,\<c,c'\>,\<d,d'\>)$ if and only if $A\models \Gamma(a,b,c,d)$ and $B \models \Gamma(a',b',c',d')$.
\end{enumerate}
\end{prop}
\begin{proof}
The first four are obvious. To check $\Gamma$ is preserved by direct
products, suppose $A\models \Gamma(a,b,c,d)$ and $B \models
\Gamma(a',b',c',d')$. Now take $\th \in FC(A\times B)$ and assume
$(\<c,c'\>,\<d,d'\>)\in\th$. By BFC, 
there exist factor congruences  $\th^0 \in FC(A)$ and $\th^1 \in
FC(B)$ such that $\th = \{(\<x,x'\>,\<y,y'\>) : (x,x')\in \th^0,
(y,y')\in\th^1\}$. This yields 
$(c,d)\in\th^0$ and $(c',d')\in\th^1$, and then we have
$(a,b)\in\th^0$ and $(a',b')\in\th^1$ by hypothesis. Hence
$(\<a,a'\>,\<b,b'\>)\in\th$ and  we have showed that $A\times B\models
  \Gamma(\<a,a'\>,\<b,b'\>,\<c,c'\>,\<d,d'\>)$. Preservation of $\Gamma$
  by direct factors is similar.
\end{proof}
It turns out that if $\Gamma$ is a first-order formula, it is the
\emph{weakest} witness for (*).  In the case of finite languages, it
can be proved that if no 
nontrivial algebra of $\V$ has a trivial 
subalgebra, then  $\Gamma$ is a first-order formula. This
is an easy consequence of~\cite{DFC}.

If one replaces $FC(A)$ in the
definition of $\Gamma$ by some other set of congruences that contains $\Delta$,
Proposition~\ref{p:Psi} will still hold with the possible exception of (5).
One nice conjecture would be that one
may obtain some first-order formula by replacing $FC(A)$ in the
definition of $\Gamma$ by some bigger set of congruences. While
this is indeed the case for semilattices,  we
cannot expect to obtain in such manner every 
formula witnessing (*), even not one that results from our
construction, as the following counterexample
shows.

Take the variety $\V$ in the language $\{0,\cdot\}$ defined by the
following identities:
\begin{align*}
(x \cdot y) \cdot z &\id x \cdot (y \cdot z) \\
x \cdot x &\id x \\
x \cdot 0 &\id 0 \cdot x \id 0.
\end{align*}
We will calculate the terms $s_i, t_i$ and $L_\alpha,R_\alpha$. For
this particular case, $N=n=2$. Define:
\begin{align*}
s_1 &:= x &  s_2 &:= y \\
t_1 &:= 0 &  t_2 &:= 0 
\end{align*}
\begin{align*}
  L_1 &:= x \cdot y_1 & L_{11} &:= z \cdot y_1 & 
L_{12} &:= y \cdot y_1 
 &  R_1 &:= y \cdot y_1 \\
 &  & 
R_{11} &:= w \cdot  y_1
& R_{12} &:= y \cdot y_1 
\end{align*}
\begin{align*}
  L_2 &:= y_2 \cdot x & L_{21} &:= y_2 \cdot z & 
L_{22} &:= y_2 \cdot y 
&  R_2 &:= y_2 \cdot y \\
 &  & 
R_{21} &:= y_2 \cdot w 
& R_{22} &:=
y_2 \cdot y  
\end{align*}
Then the formula $\pi(x,y,z,w)$ obtained for these terms is the
conjunction of $\Phi_1$ and $\Phi_2$:
\begin{multline*}
\Phi_1  := \exists y_1 \forall x_1 \exists y_2 \forall x_2:  z \cdot
y_1  = w \cdot  y_1 
\ \y\  y \cdot y_1  = y \cdot y_1 \ \y\  \\ 
\ \y\   y_2
\cdot z =  y_2 \cdot w \ \y\   y_2 \cdot y = y_2 \cdot y.
\end{multline*}
\begin{multline*}
\Phi_2 := \exists x_1 \forall y_1 \exists x_2 \forall y_2: \bigl((z \cdot
y_1  = w \cdot  y_1 \ \y\  y \cdot y_1  = y \cdot y_1) \impl  x \cdot
y_1 = y \cdot y_1 \bigr) \ \y\  \\ 
\ \y\  \bigl(( y_2
\cdot z =  y_2 \cdot w \ \y\   y_2 \cdot y=y_2 \cdot y ) \impl  y_2 \cdot x
  = y_2 \cdot y\bigr).   
\end{multline*}
Formula $\Phi_1$ holds trivially in $\V$ (take $y_1,y_2=0$) and
$\Phi_2$ may be simplified to:
\[ \forall u (z \cdot u = w \cdot u \impl x \cdot u = y
\cdot u) \y (u \cdot z = u \cdot w \impl u \cdot x = u
\cdot y).\] 

\noindent\begin{minipage}{4.5in}
Now, the algebra $A$ given by the table on the right is in $\V$. We
have $A\models\pi(a,b,a,b)$ and $A\models\pi(c,c,a,b)$. If $\pi(x,y,z,w)$ were of
the form
\[\forall \th \in FC^*(A): (z,w)\in \th \ent (x,y)\in \th,\] 
\end{minipage}
\begin{minipage}{1.7in}
\begin{center}
      \begin{tabular}{|c|cccc|}
	\hline $\cdot^A$ & 0 & $a$ & $b$ & $c$ \\\hline 
	0 & 0 & 0 & 0  & 0 \\
	$a$ & 0 & $a$ & $a$  & $a$ \\
	$b$ & 0 & $a$ & $b$  & $c$ \\
	$c$ & 0 & $c$ & $c$  & $c$ \\\hline
   \end{tabular} 
\end{center}
\end{minipage}
\smallskip\\
for some set of congruences $FC^*(A)$,  we should also have $A\models\pi(a\cdot
c,b \cdot c,a,b)$  by Proposition~\ref{p:Psi}~(1). But that's not the case since $ a \cdot a = b
\cdot a$ and  $(a\cdot c) \cdot a \neq (b\cdot c) \cdot a$.
\medskip

For the case of semilattices, the terms $L_\alpha$,  $R_\alpha$ and
$s_i$ are the same and we have to take $t_1 = t_2 := z \cdot w$. We obtain the
simpler formula $\pi_s(x,y,z,w)$:
\[ \forall u (z \cdot u = w \cdot u \impl x \cdot u = y
\cdot u),\]
which is equivalent to $\forall \th \in FC^*(A): (z,w)\in \th \ent
(x,y)\in \th$ for every semilattice $A$, where we take 
\[FC^*(A) :=
\{\th_{z,w} \in \CON(A): z,w\in A\} \text{ and } \th_{z,w} := \{(x,y) \in A^2
: \pi_s(x,y,z,w)\}.\]
\section{Acknowledgements}
I would like to thank specially Diego Vaggione, for his constant
support and for the sharpest observations. I would also like to thank
Ross Willard for his generous contribution to this paper. Finally,
Teresita Terraf was very helpful with the details of presentation.

\bigskip
\bigskip

\begin{quote}
CIEM --- Facultad de Matem\'atica, Astronom\'{\i}a y F\'{\i}sica 
(Fa.M.A.F.) 

Universidad Nacional de C\'ordoba - Ciudad Universitaria

C\'ordoba 5000. Argentina.

email: \texttt{sterraf@mate.uncor.edu}
\end{quote}
\end{document}